\documentclass[11pt,a4paper]{article}

\usepackage[T1]{fontenc}
\usepackage[margin=1in]{geometry}
\usepackage{amsmath,amssymb,amsthm}
\usepackage{array}
\usepackage{booktabs}
\usepackage{float}
\usepackage{xcolor}
\usepackage{url}
\usepackage[hidelinks]{hyperref}

\newtheorem{theorem}{Theorem}[section]
\newtheorem{lemma}[theorem]{Lemma}
\newtheorem{corollary}[theorem]{Corollary}
\newtheorem{proposition}[theorem]{Proposition}
\newtheorem{conjecture}[theorem]{Conjecture}

\newcommand{\N}{\mathcal N}
\newcommand{\om}{\omega}
\newcommand{\dunion}{\mathbin{\dot\cup}}

\title{An infinite family of minimally nonperfectly divisible graphs
with a bisimplicial vertex}
\author{Lizhong Chen\\
Department of Mathematics, The Hong Kong University of Science and
Technology\\
Clear Water Bay, Kowloon, Hong Kong\\
\texttt{lchendh@connect.ust.hk}}
\date{}

\hypersetup{
  pdftitle={An infinite family of minimally nonperfectly divisible graphs with a bisimplicial vertex},
  pdfauthor={Lizhong Chen}
}

\begin{document}

\maketitle

\begin{abstract}
We construct an explicit infinite family of minimally nonperfectly
divisible graphs \(G_t\), each with a bisimplicial vertex \(v_t\) of degree
four and a cutvertex. Thus the same family disproves Ho\`ang's conjectures
excluding bisimplicial vertices and clique cutsets from such graphs.
For every integer \(t\geq1\), the graph \(G_t\) has clique number three and
satisfies
\[
 |V(G_t)|=93+30(t-1),\qquad |E(G_t)|=320+104(t-1).
\]
Deleting \(v_t\) gives a perfectly divisible graph \(H_t\) with an edge
whose endpoints both lie outside the perfect part of every perfect
division. This gives a negative answer to the prescribed-vertex question
of Hu, Xu and Zhuang. The same graphs answer their second question
affirmatively: each \(H_t\) is perfectly divisible but not perfectly
weight divisible. The cutvertex and weighted consequences were pointed
out by Hu, Xu and Zhuang, who also observed that replacing one endpoint
of the edge by two adjacent twins gives a minimally nonperfectly
divisible graph with a homogeneous set.
\end{abstract}

\medskip
\noindent\textbf{Keywords.}
Perfect divisibility; minimally nonperfectly divisible graph;
bisimplicial vertex; rooted product; perfect graph.

\smallskip
\noindent\textbf{Mathematics Subject Classification (2020).}
05C17; 05C15; 05C75.

\section{Introduction}

All graphs are finite and simple. We follow West \cite{West2001} for
graph-theoretic terminology. For a graph \(G\), let \(\chi(G)\) and
\(\om(G)\) denote its chromatic number
and clique number, respectively. A graph \(G\) is perfect if
\(\chi(H)=\om(H)\) for every induced subgraph \(H\) of \(G\).

Ho\`ang introduced perfect divisibility in the study of \(\chi\)-bounded
graph classes \cite{Hoang2018}. Chudnovsky and Sivaraman proved that
several graph classes are perfectly divisible and observed that perfect
divisibility gives a quadratic bound on the chromatic number in terms of
the clique number \cite{ChudnovskySivaraman2019}. Following Hu, Xu and
Zhuang \cite{HuXuZhuang2026}, a \emph{perfect division} of a graph \(H\)
is a partition
\[
 V(H)=A\dunion B
\]
such that \(H[A]\) is perfect and
\[
\om(H[B])<\om(H).
\]
We call \(A\) the \emph{perfect part}. Ho\`ang calls such a partition a
\emph{good partition} \cite{Hoang2026}. A graph \(G\)
is perfectly divisible if every induced subgraph of \(G\) with at least
one edge has a perfect division. Equivalently, one may require a perfect
division for every nonempty induced subgraph. Indeed, with the convention
\(\om(\varnothing)=0\), a nonempty edgeless graph \(H\) has the perfect
division \((V(H),\varnothing)\).

A graph is minimally nonperfectly divisible (an MNPD graph) if it is not
perfectly divisible but each of its proper induced subgraphs is perfectly
divisible \cite{Hoang2026,HuXuZhuang2026}. A vertex \(v\) is
\emph{bisimplicial} if its neighbourhood is the union of two cliques
\cite{ChudnovskySeymour2023}. Chudnovsky and Seymour proved that every
nonempty even-hole-free graph contains a bisimplicial vertex
\cite[Theorem~1.1]{ChudnovskySeymour2023}. For MNPD graphs, Ho\`ang
proposed the following conjecture.

\begin{conjecture}\label{conj:bisimplicial}
No MNPD graph contains a bisimplicial vertex
\cite[Conjecture 4.5]{Hoang2026}.
\end{conjecture}

Ho\`ang proved that no MNPD graph has a simplicial vertex, that is, a
vertex whose neighbourhood is a clique \cite[Lemma~2.2]{Hoang2026}.
Conjecture~\ref{conj:bisimplicial} asks whether the same conclusion
holds when two cliques cover the neighbourhood. The following family
gives a negative answer.

\begin{theorem}\label{thm:main}
There is an explicit sequence \((G_t,v_t)\), \(t=1,2,\ldots\), of pairwise
nonisomorphic graphs such that, for every \(t\geq1\),
\[
 G_t\text{ is MNPD},\qquad \om(G_t)=3,
\]
and \(v_t\) is a bisimplicial vertex of degree four. Moreover,
\[
 |V(G_t)|=93+30(t-1),\qquad
 |E(G_t)|=320+104(t-1).
\]
\end{theorem}

Deleting \(v_t\) gives a perfectly divisible graph \(H_t=G_t-v_t\).
An edge of \(H_t\) will be useful for both questions about perfect divisions.

\begin{corollary}\label{cor:intro-prescribed}
There is an explicit sequence \((H_t)_{t\geq1}\) of pairwise
nonisomorphic perfectly divisible graphs such that each \(H_t\) contains
adjacent vertices \(b,c\) with the following property: in every perfect
division \((A,B)\) of \(H_t\),
\[
 b,c\in B.
\]
In particular, there need not exist a perfect division with a prescribed
vertex in its perfect part.
\end{corollary}

The construction gives the following answers to two conjectures of
Ho\`ang and three questions of Hu, Xu and Zhuang.
\begin{enumerate}
\item \emph{Bisimplicial vertices.}
Ho\`ang conjectured that no MNPD graph contains a bisimplicial vertex
\cite[Conjecture~4.5]{Hoang2026}. The vertex \(v_t\) of degree four makes each \(G_t\) a counterexample.
\item \emph{Clique cutsets.}
Ho\`ang conjectured that no MNPD graph has a clique cutset
\cite[Conjecture~1.2]{Hoang2026}. Each \(G_t\) also disproves this
conjecture, since it has a cutvertex.
\item \emph{A prescribed vertex in the perfect part.}
Can every vertex of a perfectly divisible graph be placed in the perfect
part of some perfect division \cite[Problem~4.1]{HuXuZhuang2026}?
The answer is negative. In every perfect division of \(H_t\), both
endpoints of \(bc\) lie outside the perfect part.
\item \emph{Perfect weight divisibility.}
Is there a perfectly divisible graph that is not perfectly weight
divisible \cite[Problem~4.2]{HuXuZhuang2026}? Each \(H_t\) gives an
affirmative answer. Giving \(b\) weight two and every other vertex
weight one rules out a perfect division for this weighting.
\item \emph{Homogeneous sets.}
Can an MNPD graph contain a homogeneous set
\cite[Introduction]{HuXuZhuang2026}? The answer is affirmative.
Replacing \(b\) in \(H_t\) by two adjacent twins gives an MNPD graph
in which the twins form a homogeneous set.
\end{enumerate}
Hu, Xu and Zhuang pointed out the cutvertex, weighting and twin
observations in (ii), (iv) and (v) \cite[Section~4]{HuXuZhuang2026}.
We give the weighting argument in Section~\ref{sec:prescribed}.

To construct \(G_t\), we start with a tree and a triangle. Each vertex
of the triangle has two neighbours in the tree and forms an odd hole with
the path between them. At each tree vertex we attach a copy of a fixed
15-vertex graph \(F\), identifying its root \(r\) with that vertex.
Removing the root separates the rest of its copy from the tree, which
explains the cutvertices. Now suppose that \(G_t\) has a perfect
division. Since \(F\) and \(G_t\) both have clique number three, the
division restricts to a perfect division of each copy of \(F\). Every
such division puts the root in the perfect part. With all tree vertices
in that part, none of the three odd holes can have its remaining vertex
there as well. This puts the entire triangle in the other part, a
contradiction.

The proof that \(F\) forces its root into the perfect part turns on a
5-cycle in the neighbourhood of \(r\). Placing \(r\) in the other part
would give a 3-colouring of its neighbourhood. The adjacencies around the cycle, together with a recolouring argument,
show that three specified pairs of vertices must be monochromatic,
one pair of each colour. Each choice of the pair outside the perfect
part yields an induced 5-cycle in the perfect part or a triangle in
the other part.

Minimality depends on what happens when a copy of \(F\) loses a vertex.
In a proper induced subgraph of \(F\) containing \(r\) and an edge, we can place
\(r\) in either part. When the clique number is three, four explicit
partitions suffice to place \(r\) outside the perfect part. When it is
two, a 3-colouring gives both choices. Section~\ref{sec:preliminaries}
proves these properties of \(F\). The rooted product lemma in
Section~\ref{sec:product} shows how to combine divisions that agree at
the roots. Section~\ref{sec:auxiliary} constructs the required partitions of
the base graph, and Section~\ref{sec:family} applies the lemma to prove minimality.

\section{Preliminaries and the rooted graph}
\label{sec:preliminaries}\label{sec:rooted}

For \(X\subseteq V(G)\), let \(G[X]\) denote the subgraph of \(G\)
induced by \(X\). The neighbourhood of a vertex \(v\) is denoted by
\(N_G(v)\), or simply \(N(v)\) when \(G\) is clear from context. A set
of pairwise nonadjacent vertices is called an independent set. A hole is
an induced cycle of length at least four, and an antihole is the complement
of a hole. An odd hole is a hole of odd length, and an odd antihole is the
complement of an odd hole.

The Strong Perfect Graph Theorem of Chudnovsky, Robertson, Seymour
and Thomas characterises perfection in terms of these two obstructions
\cite{ChudnovskyRobertsonSeymourThomas2006}.

\begin{theorem}[Strong Perfect Graph Theorem]\label{thm:spgt}
A graph is perfect if and only if it contains no odd hole and no odd
antihole.
\end{theorem}

A \emph{rooted graph} is a pair \((D,z)\), where \(D\) is a graph and
\(z\in V(D)\) is its \emph{root}. Let \(Q\) be a labelled graph with
vertices \(q_1,\ldots,q_n\), and let
\((D_1,z_1),\ldots,(D_n,z_n)\) be rooted graphs on pairwise disjoint
vertex sets, all disjoint from \(V(Q)\). The \emph{rooted product} of
\(Q\) by this sequence identifies \(z_i\) with \(q_i\) for each \(i\)
and adds no other edges \cite[Definition~1.1]{GodsilMcKay1978}. We call
the image of \(D_i\) the \emph{rooted factor} at \(q_i\) and continue
to denote it by \(D_i\). A factor may consist of a single vertex.

We use the fact that pasting perfect graphs along a clique preserves
perfection \cite{Trotignon2015}. Suppose two perfect graphs \(G_1\) and
\(G_2\) meet in a common clique, possibly empty, and no edge joins a
vertex of \(V(G_1)\setminus V(G_2)\) to a vertex of \(V(G_2)\setminus V(G_1)\).
Then \(G_1\cup G_2\) is perfect.

Let \(F\) be the graph on vertex set \(\{0,1,\ldots,14\}\) whose
encoding in the \texttt{graph6} format \cite{McKayPiperno2014} is
\[
 \verb|Nhru`dwjS_yLMeF@bv?|,
\]
and let \(r=3\). The complete adjacency list of \(F\) is given in
Appendix~\ref{app:rooted-adjacency}.

\begin{lemma}\label{lem:rooted-factor}
The rooted graph \((F,r)\) has the following properties.
\begin{enumerate}
  \item The graph \(F\) has 15 vertices, 51 edges, clique number three,
  and is imperfect and perfectly divisible.
  \item In every perfect division \((A,B)\) of \(F\), the root belongs
  to the perfect part \(A\).
  \item If \(X\subsetneq V(F)\), \(r\in X\), and \(F[X]\) has an edge,
  then \(F[X]\) has a perfect division with \(r\) in the perfect part
  and another perfect division with \(r\) in the other part.
\end{enumerate}
\end{lemma}

\begin{proof}
We first describe \(F\) from the root outwards. Write \(\N=N_F(3)\).
To construct \(F[\N]\), start with the induced 5-cycle
\[
     11\,-\,13\,-\,12\,-\,2\,-\,7\,-\,11 .
\]
Add four vertices \(10,5,6,8\), corresponding respectively to
\(13,12,2,7\). Join each new vertex to the two cycle neighbours of
its corresponding vertex. Finally, join vertex 14 to all four new
vertices. This completes \(F[\N]\). The root 3 is adjacent to all of
\(\N\).
The other four vertices \(D=\{0,1,4,9\}\) induce a diamond with missing
edge \(1\,9\). Their neighbours in \(\N\) are
\[
\begin{array}{c|l@{\qquad}c|l}
0&5,6,10,11 &1&2,5,6,8,11\\
4&8,12,13,14&9&2,5,7,10,12,14.
\end{array}
\]
Counting the edges within \(\N\), from 3 to \(\N\), within \(D\), and
from \(D\) to \(\N\) gives \(17+10+5+19=51\). The induced 5-cycle shows
that \(F\) is imperfect. To find its clique number, observe that \(D\)
contains a triangle whereas \(\N\) is triangle-free. Thus no 4-clique
contains 3. A 4-clique avoiding 3 must contain an edge of \(D\), since \(\N\)
is triangle-free. The common neighbours in \(\N\) of the edges
\(0\!-\!1,0\!-\!4,0\!-\!9,1\!-\!4,4\! -\!9\) of \(D\)
are, respectively,
\[
 \{5,6,11\},\quad \varnothing,\quad \{5,10\},\quad
 \{8\},\quad \{12,14\}.
\]
Each set is independent, and both triangles of \(D\) contain the edge
\(0\! -\!4\). These facts exclude a 4-clique. Hence \(\omega(F)=3\).

\smallskip\noindent\emph{Why the root is forced.}
Suppose that a perfect division \((A,B)\) of \(F\) has \(3\in B\).
Since \(F[B]\) is triangle-free, no two neighbours of 3 in \(B\)
are adjacent. On the other hand, \(F[A\cap\N]\) is perfect and
triangle-free, hence bipartite. Colour \(A\cap\N\) with two colours,
and give \(B\cap\N\) a third colour, red.

In this 3-colouring, vertices 11 and 14 must have the same colour.
Suppose otherwise. We could then eliminate the colour of 14 from the
5-cycle \((11,13,12,2,7)\). Each cycle vertex of that colour is one of
\(13,12,2,7\). Give it the colour of the corresponding vertex in
\(10,5,6,8\), which is adjacent to 14 and to both neighbours of the
cycle vertex. The recoloured vertices were pairwise nonadjacent, so the
cycle would now be properly coloured with just two colours.

Each of \(5,7,6,13\) is adjacent to 11 or 14 and therefore avoids their
common colour. The edges \(5\,13\) and \(7\,6\) each use the other two
colours. If 5 and 7 had different colours, so would 13 and 6. Vertex 2
is adjacent to 5 and 7, and vertex 12 to 13 and 6. Both 2 and 12 would
then need the colour of 11, although they are adjacent. Thus the three pairs
\[
             \{5,7\},\qquad \{6,13\},\qquad \{11,14\}
\]
are monochromatic, one pair of each colour. We consider the three
possibilities for the red pair. In each case, the cycles listed below
are induced 5-cycles, with vertices listed in cyclic order. Each must
meet \(B\), just as each triangle must meet \(A\).
\begin{itemize}
\item If \(5,7\) are red, then \(6,11,13,14\in A\). Vertex 2
also lies in \(A\), since it is adjacent to 5 and the red vertices
are independent.
If \(1\in A\), the cycles \((1,2,12,13,11)\) and
\((1,6,14,10,11)\) put 12 and 10 in \(B\), forming a triangle with 3.
Hence \(1\in B\). Since \(1,5\in B\), the triangle \(0,1,5\)
forces \(0\in A\).
The cycles \((0,6,14,8,11)\) and \((0,6,12,13,11)\) now put
8 and 12 in \(B\). The triangles \(1,4,8\) and \(3,10,12\)
put 4 and 10 in \(A\), completing the cycle \((4,13,11,10,14)\) in \(A\).
\item If \(6,13\) are red, then \(5,7,11,14\in A\).
If \(0\in A\), the cycles \((0,5,2,7,11)\) and
\((0,5,14,8,11)\) put 2 and 8 in \(B\), forming a triangle with 3.
Thus \(0\in B\), and the triangle \(0,1,6\) puts \(1\in A\).
To meet the cycles \((1,5,9,7,11)\) and \((1,5,14,10,11)\),
\(B\) must contain 9 and 10. But then \(0,9,10\) form a triangle in \(B\).
\item If \(11,14\) are red, then \(5,6,7,13\in A\). Vertex 8 is
adjacent to 14, so it too lies in \(A\).
If \(0\in A\), the cycles \((0,5,2,7,6)\) and
\((0,5,13,12,6)\) put 2 and 12 in \(B\), forming a triangle with 3.
Thus \(0\in B\), and the triangle \(0,1,11\) puts \(1\in A\).
The cycles \((1,5,9,7,6)\) and \((1,5,13,12,6)\) put 9 and 12 in \(B\).
The triangles \(0,4,9\) and \(2,3,12\) therefore force 4 and 2
into \(A\). Now \(A\) contains the induced cycle \((2,5,13,4,8)\).
\end{itemize}
All three cases give a contradiction. Thus \(3\in A\).

\smallskip\noindent\emph{Divisions at clique number three.}
We first give a division of \(F\) that also works, by restriction, for
every induced subgraph of clique number three. Take
\[
    A=\{1,3,9\}\cup\{0,7,13,14\},\qquad B=V(F)\setminus A .
\]
The two sets in the expression for \(A\) are independent, so \(F[A]\)
is bipartite. In \(F[B]\), vertices 5 and 6 have degree one. Deleting
them leaves five independent sets, listed here in cyclic order,
\(\{2,4\}\), \(\{8\}\), \(\{11\}\), \(\{10\}\) and \(\{12\}\). Edges
run only between consecutive sets, so \(F[B]\) is triangle-free.
Restricting this division to an induced subgraph of clique number three
keeps the root in the perfect part whenever it is present.

Now suppose that a nonroot vertex \(z\) is missing and we want \(3\in B\).
Choose a row of the table whose cycle contains \(z\). In that row,
\(A=V(F)\setminus B\), and deleting \(z\) breaks the indicated cycle
into a path. The third column records how to reduce \(F[A]\) to the
cycle. Read the entries from left to right. After each vertex we list
its neighbours among the other vertices still present. For a starred
vertex we list its nonneighbours instead.
\begin{center}\small
\begin{tabular}{@{}lll@{}}
\toprule
\(B\)&cycle in \(A\)&removals in \(A\)\\
\midrule
\(1,3,5,7,8,9,10\) & \(0,6,12,13,11\) &
\(2\{12\},\ 4^*\{6,11\},\ 14\{6\}\)\\
\(0,3,4,6,8,10,13\) & \(1,5,9,7,11\) &
\(2^*\{11,14\},\ 12\{9\},\ 14\{5,9\}\)\\
\(1,3,4,5,6,9,11\) & \(2,8,14,10,12\) &
\(0\{10\},\ 7\{2\},\ 13\{12\}\)\\
\(1,3,5,7,8,9,12\) & \(4,13,11,10,14\) &
\(0^*\{2,13,14\},\ 2\varnothing,\ 6\{14\}\)\\
\bottomrule
\end{tabular}
\end{center}
In rows one, three and four, \(F[B]\) is bipartite with one side
\(\{1,3,9\}\). In row two, the only edges run between consecutive
sets in the cyclic list
\(\{0\},\{4\},\{8,13\},\{3\},\{6,10\}\), each set independent.
Thus every row has \(3\in B\) and \(B\) triangle-free.

To see why each row works, consider restoring the vertices of the
third column in reverse order. At the time of removal, each unstarred
vertex has a clique as its neighbourhood, and the nonneighbours of
each starred vertex form an independent set. Deleting \(z\) preserves
both properties.

In either case, adding the vertex back to a perfect graph preserves
perfection. To see this, consider any induced subgraph \(J\) containing
the restored vertex \(x\). If its neighbours form a clique, they use at
most \(\omega(J)-1\) colours in an optimal colouring of \(J-x\), leaving
a colour for \(x\) in a palette of \(\omega(J)\) colours. If its
nonneighbours form an independent set, give them and \(x\) one colour.
The remaining vertices induce a perfect graph on \(N_J(x)\) with clique
number at most \(\omega(J)-1\).

Starting from the path left by deleting \(z\) and restoring the three
vertices therefore gives a perfect graph \(F[A\setminus\{z\}]\).
The four cycles together contain every nonroot vertex. Thus for every
proper \(X\) containing \(r\) with \(\omega(F[X])=3\), we can choose
\(z\notin X\) and a row whose cycle contains \(z\). Restricting that
row's partition to \(X\) gives a perfect division with \(r\in B\).

\smallskip\noindent\emph{Divisions at clique number two.}
The following five disjoint triangles cover \(V(F)\).
\[
 (0,10,11),\quad(1,2,8),\quad(3,6,7),\quad(4,12,13),\quad(5,9,14).
\]
A triangle-free induced subgraph contains at most two vertices of each
triangle, hence at most ten vertices in all. By a theorem of
Chv\'atal~\cite{Chvatal}, every triangle-free graph on at most ten
vertices is 3-colourable (see Appendix~\ref{app:ten} for a short proof).
If the subgraph has an edge, taking one colour class for \(B\) and the
other two for \(A\) gives a perfect division. If the root is present, choosing its colour class for \(B\) puts the
root in \(B\). Choosing either other class puts the root in \(A\).
\end{proof}

\section{A rooted product lemma}\label{sec:product}

The factors in the next lemma have the two properties of \(F\) that
matter at a root. Every perfect division of an intact \(D_i\) places its root in the
perfect part. If a proper induced subgraph of \(D_i\) contains the root
and an edge, it admits a perfect division with the root in either part.
We therefore partition the base first, keeping track of which \(D_i\)
are present in full.

\begin{lemma}\label{lem:rooted-product}
Let \(Q^\ast\) be a graph, let
\(R=\{q_1,\ldots,q_m\}\subseteq V(Q^\ast)\), and let
\((D_1,z_1),\ldots,(D_m,z_m)\) be rooted graphs whose vertex sets are
pairwise disjoint and disjoint from \(V(Q^\ast)\). Let \(C\) be obtained by
identifying \(z_i\) with \(q_i\) for every \(i\), with a single-vertex
factor at each vertex of \(V(Q^\ast)\setminus R\). Suppose that
\(\om(Q^\ast)\leq3\) and the following conditions hold.
\begin{enumerate}
  \item For every \(i\), the graph \(D_i\) has clique number three, is
  perfectly divisible, and every perfect division of \(D_i\) puts \(z_i\)
  in the perfect part.
  \item If \(X\subsetneq V(D_i)\), \(z_i\in X\), and \(D_i[X]\) has an
  edge, then \(D_i[X]\) has a perfect division putting \(z_i\) in either
  prescribed part.
  \item For every \(Z\subseteq V(Q^\ast)\) and
  \(M\subseteq Z\cap R\), there is a partition
  \(Z=A_Q\dunion B_Q\) such that
  \[
    Q^\ast[A_Q]\text{ is perfect},\qquad
    \om(Q^\ast[B_Q])\leq2,\qquad M\subseteq A_Q,
  \]
  except when \(Z=V(Q^\ast)\) and \(M=R\).
  \item For every \(Z\subseteq V(Q^\ast)\), there is a partition
  \(Z=A_Q\dunion B_Q\) such that \(Q^\ast[A_Q]\) is perfect and
  \(B_Q\) is an independent set.
\end{enumerate}
Then every proper induced subgraph of \(C\) with at least one edge has a
perfect division.
\end{lemma}

\begin{proof}
Let \(J\) be a proper induced subgraph of \(C\) with an edge. Put
\(Z=V(J)\cap V(Q^\ast)\) and \(J_i=J[V(J)\cap V(D_i)]\). The set
\(M=\{q_i:J_i=D_i\}\) consists of the roots whose factors are present
in full. Every clique of \(J\) lies in the base or in one factor, so
\(k=\om(J)\) is either two or three.

We first partition \(Z=A_Q\dunion B_Q\) so that \(Q^\ast[A_Q]\) is
perfect, \(\om(Q^\ast[B_Q])<k\), and \(M\subseteq A_Q\). If \(k=3\),
condition (iii) applies. Its exception would require the entire base
and every factor to be present, contrary to \(J\ne C\). If \(k=2\), none
of the \(D_i\) is present in full, since each contains a triangle. Thus
\(M=\varnothing\) and condition (iv) gives the partition.

We extend the partition of \(Z\) to each \(J_i\). If \(J_i=D_i\), its
root already lies in \(A_Q\), so any perfect division of \(D_i\) agrees
with the base by (i). If \(J_i\) is proper and has an edge, (ii) lets us
match the choice at the root whenever the root is present. When the
root is absent, any perfect division of \(J_i\) will do, and one exists
by (i). If \(J_i\) is edgeless, we keep the root in its assigned part
when present and put all other vertices of \(J_i\) in the perfect part.

Let \(A\) and \(B\) be the resulting parts of \(V(J)\). Each perfect
part from a factor meets the perfect part of the base in at most its
root, so pasting them together preserves perfection. It remains to
check that \(B\) contains no \(k\)-clique. On the base this holds by
the choice of \(B_Q\). In each \(J_i\) with an edge, the clique number
of the part in \(B\) is strictly less than \(\om(J_i)\leq k\).
An edgeless \(J_i\) contains no \(k\)-clique since \(k\geq2\).
Every clique of \(J\) lies in the base or in one factor, so \((A,B)\)
is a perfect division.
\end{proof}

\section{The auxiliary graphs
\texorpdfstring{\(\boldsymbol{Q_t}\)}{Qt}}\label{sec:auxiliary}

Fix an integer \(t\geq1\). Let \(P_t\) be an \(x\)-\(y\) path of length
\(2t-1\), whose internal vertices in order are
\(p_1,\ldots,p_{2t-2}\). This set is empty when \(t=1\), so
\(P_1=xy\). Define a tree \(T_t\) with vertex set
\[
 R_t=\{a,d,u,w\}\cup V(P_t)
\]
and edge set
\[
 E(T_t)=E(P_t)\cup\{ax,dy,yu,xw\}.
\]
Thus \(a,w\) are leaves at \(x\), and \(d,u\) are leaves at \(y\).

To form \(Q_t\), add a triangle on \(K=\{b,c,v\}\) and join it to
\(T_t\) by the six edges
\begin{equation}\label{eq:attachment-edges}
 ab,\qquad ub,\qquad dc,\qquad wc,\qquad av,\qquad dv.
\end{equation}
Figure~\ref{fig:Qt} shows \(Q_t\). We will attach a copy of \((F,r)\)
at each vertex of the tree, that is, at each vertex of \(R_t\).

\begin{figure}[H]
\centering
\setlength{\unitlength}{0.45mm}
\begin{picture}(230,165)(-5,0)
  % Edges of the tree T_t other than the variable path P_t.
  \qbezier(78,65)(94,37.5)(110,10)        % ax
  \qbezier(130,105)(120,90)(110,75)       % dy
  \qbezier(110,75)(94,86.5)(78,98)        % yu
  \qbezier(110,10)(126,37.5)(142,65)      % xw
  % The x-y path P_t; the break represents its omitted internal vertices.
  \qbezier(110,10)(110,25)(110,40)
  \qbezier(110,55)(110,65)(110,75)
  \put(106.5,41){\(\vdots\)}
  \put(116,43){\(P_t\)}

  % Edges from b and c to the tree.
  \qbezier(45,120)(61.5,92.5)(78,65)      % ab
  \qbezier(45,120)(61.5,109)(78,98)       % ub
  \qbezier(175,120)(152.5,112.5)(130,105) % dc
  \qbezier(142,65)(158.5,92.5)(175,120)   % wc

  % The triangle K={b,c,v} and the remaining attachment edges.
  \qbezier(45,120)(110,120)(175,120)       % bc
  \qbezier(110,150)(77.5,135)(45,120)      % vb
  \qbezier(110,150)(142.5,135)(175,120)    % vc
  \qbezier(110,150)(-20,150)(78,65)        % va
  % The white gap distinguishes the crossing of bc and vd.
  \put(123.33,120){{\color{white}\circle*{5}}}
  \qbezier(110,150)(120,127.5)(130,105)    % vd

  % Vertices of R_t are open when displayed; b,c,v are filled.
  \put(78,65){\circle{6}}
  \put(130,105){\circle{6}}
  \put(110,10){\circle{6}}
  \put(110,75){\circle{6}}
  \put(78,98){\circle{6}}
  \put(142,65){\circle{6}}
  \put(45,120){\circle*{5}}
  \put(175,120){\circle*{5}}
  \put(110,150){\circle*{5}}

  \put(68,58){\(a\)}
  \put(134,108){\(d\)}
  \put(105,-1){\(x\)}
  \put(114,73){\(y\)}
  \put(81,103){\(u\)}
  \put(146,58){\(w\)}
  \put(31,124){\(b\)}
  \put(180,124){\(c\)}
  \put(114,151){\(v\)}
\end{picture}
\caption{The auxiliary graph \(Q_t\). Open circles mark vertices of the
tree \(T_t\), and filled circles mark the triangle \(K\). The vertical
path \(P_t\) has length \(2t-1\). Its omitted internal vertices also
belong to \(R_t\). For \(t=1\), this path is the edge \(xy\).
The gap where \(bc\) crosses \(vd\) indicates that the edges do not meet.}
\label{fig:Qt}
\end{figure}

Each \(z\in K\) has exactly two neighbours in the tree. Let \(P_z\)
be the unique path in \(T_t\) between them. Together with \(z\),
this path induces the cycle \(C_z=Q_t[V(P_z)\cup\{z\}]\), with vertex
set given by
\begin{equation}\label{eq:odd-holes}
\begin{split}
 V(C_b)&=V(P_t)\cup\{a,u,b\},\\
 V(C_c)&=V(P_t)\cup\{d,w,c\},\\
 V(C_v)&=V(P_t)\cup\{a,d,v\}.
\end{split}
\end{equation}
Each \(C_z\) is an induced cycle of length \(2t+3\), and hence is an odd
hole.

We now prove the two assignment properties required by Lemma
\ref{lem:rooted-product}. We use the standard facts that forests and
bipartite graphs are perfect \cite[Section 8.1]{West2001}.

\begin{lemma}\label{lem:assignments}
For every integer \(t\geq1\), the graph \(Q_t\) and the set \(R_t\) have
the following properties.
\begin{enumerate}
  \item For every \(Z\subseteq V(Q_t)\) and
  \(M\subseteq Z\cap R_t\), there is a partition
  \(Z=A_Q\dunion B_Q\) such that
  \[
    Q_t[A_Q]\text{ is perfect},\qquad
    \om(Q_t[B_Q])\leq2,\qquad M\subseteq A_Q,
  \]
  except when \(Z=V(Q_t)\) and \(M=R_t\).
  \item For every \(Z\subseteq V(Q_t)\), there is a partition
  \(Z=A_Q\dunion B_Q\) such that \(Q_t[A_Q]\) is perfect and
  \(B_Q\) is an independent set.
\end{enumerate}
\end{lemma}

\begin{proof}
For (i), if \(K\not\subseteq Z\), take \(A_Q=Z\cap R_t\) and
\(B_Q=Z\cap K\). The first set induces a forest and the second has at
most two vertices.

Suppose now that \(K\subseteq Z\). Unless \(Z=V(Q_t)\) and \(M=R_t\),
we can choose \(q\in R_t\setminus M\). We want to place one vertex
\(z\) of the triangle in \(A_Q\). We can do so if \(q\) lies on
\(C_z\), since we are free to leave \(q\) out of \(A_Q\) whether or
not it belongs to \(Z\). Take \(z=b\) for \(q\in\{a,u\}\), \(z=c\) for
\(q\in\{d,w\}\), and \(z=v\) for \(q\in V(P_t)\). Put
\[
 A_Q=Z\cap\bigl((R_t\setminus\{q\})\cup\{z\}\bigr),\qquad
 B_Q=Z\setminus A_Q .
\]
The only cycle in \(Q_t[R_t\cup\{z\}]\) is \(C_z\), and it contains
\(q\) by \eqref{eq:odd-holes}. Thus \(Q_t[A_Q]\) is a forest, and
\(M\subseteq A_Q\). A triangle in \(B_Q\) would have to consist of
\(q\) and the two vertices of \(K\setminus\{z\}\). If \(q\) is a leaf
of \(T_t\), it is adjacent to \(z\) and has at most one other neighbour
in \(K\). If \(q\) lies on \(P_t\), it has no neighbour in \(K\).
In either case \(B_Q\) contains no triangle.

The exception is necessary. If all of \(R_t\) must lie in \(A_Q\),
then for each \(z\in K\) the other vertices of \(C_z\) already lie in
\(A_Q\). Perfection forces every \(z\) into \(B_Q\), which would then
contain the triangle \(K\).

For (ii), \(\{u,w,v\}\) is independent. Deleting it leaves the cycle
formed by the \(a\)-\(d\) path in \(T_t\) and the path \(d\)-\(c\)-\(b\)-\(a\).
Its length is \(2t+4\), so it is bipartite. Intersecting these two
sets with \(Z\) gives an independent set \(B_Q\) and a bipartite,
hence perfect, graph on \(A_Q\).
\end{proof}

The even cycle and the independent set give a 3-colouring of \(Q_t\), while
\(b,c,v\) induce a triangle. Therefore
\begin{equation}\label{eq:Q-clique-number}
 \om(Q_t)=3.
\end{equation}

\section{The infinite family}\label{sec:family}

For each \(t\geq1\), form \(G_t\) by attaching a separate copy of
\((F,r)\) at each vertex of \(R_t\), identifying the root with that
vertex. In this rooted product the base is \(Q_t\), and each vertex of
\(K\) has a single-vertex factor. Write \(v_t\) for the vertex \(v\)
of \(Q_t\) in \(G_t\).

There are \(|R_t|=2t+4\) attached copies of \(F\). The base graph has
\[
 |V(Q_t)|=2t+7,\qquad |E(Q_t)|=2t+12.
\]
Each copy adds 14 new vertices and 51 edges. Hence
\begin{equation}\label{eq:parameters}
\begin{split}
 |V(G_t)|&=(2t+7)+14(2t+4)
           =93+30(t-1),\\
 |E(G_t)|&=(2t+12)+51(2t+4)
           =320+104(t-1).
\end{split}
\end{equation}

\begin{proposition}\label{prop:parameters}
For every \(t\geq1\), the graph \(G_t\) has clique number three, and
\(v_t\) is a bisimplicial vertex of degree four.
\end{proposition}

\begin{proof}
A nonroot vertex of an attached copy of \(F\) has no neighbour outside
that copy. Thus every clique of \(G_t\) lies in \(Q_t\) or in one copy of
\(F\). Lemma \ref{lem:rooted-factor} and
\eqref{eq:Q-clique-number} give \(\om(G_t)=3\).

The only neighbours of \(v_t\) are \(a,b,c,d\), and they induce the path
\(a-b-c-d\). Therefore
\[
 N_{G_t}(v_t)=\{a,b\}\cup\{c,d\}
\]
is a union of two cliques. Hence \(v_t\) is bisimplicial and has degree
four.
\end{proof}

\begin{lemma}\label{lem:no-division}
For every \(t\geq1\), the graph \(G_t\) has no perfect division.
\end{lemma}

\begin{proof}
Suppose \((A,B)\) is a perfect division of \(G_t\). Since both \(G_t\)
and \(F\) have clique number three, the division restricts to a perfect
division of each attached copy of \(F\). Lemma~\ref{lem:rooted-factor}
therefore puts every vertex of \(R_t\) in \(A\). As in the exceptional
case of Lemma~\ref{lem:assignments}(i), each of the three odd holes now
forces its vertex of \(K\) into \(B\). But \(B\) cannot contain the
triangle \(K\).
\end{proof}

\begin{lemma}\label{lem:proper-division}
For every \(t\geq1\), every proper induced subgraph of \(G_t\) with at
least one edge has a perfect division.
\end{lemma}

\begin{proof}
The base graph \(Q_t\) has clique number three by
\eqref{eq:Q-clique-number}. For the copies of \(F\) attached at \(R_t\),
Lemma~\ref{lem:rooted-factor} gives conditions (i) and (ii) of
Lemma~\ref{lem:rooted-product}. The two partitions of the base in
Lemma~\ref{lem:assignments} give (iii) and (iv), so the rooted product
lemma applies.
\end{proof}

\begin{proof}[Proof of Theorem \ref{thm:main}]
Fix \(t\geq1\). Lemma \ref{lem:no-division} shows that \(G_t\) is not
perfectly divisible. Let \(H\) be a proper induced subgraph of \(G_t\),
and let \(J\) be an induced subgraph of \(H\) with at least one edge.
Then \(J\) is a proper induced subgraph of \(G_t\), so Lemma
\ref{lem:proper-division} gives a perfect division of \(J\). Hence \(H\)
is perfectly divisible, and therefore \(G_t\) is MNPD.

Proposition~\ref{prop:parameters} gives the clique number and the
degree-four bisimplicial vertex. The order and size are given by
\eqref{eq:parameters}. The orders increase strictly with \(t\), so the
graphs \(G_t\) are pairwise nonisomorphic.
\end{proof}

When \(t=1\), the path \(P_1\) is the edge \(xy\), and \(Q_1\) has nine
vertices and 14 edges. Attaching six copies of \(F\) gives
\[
 |V(G_1)|=93,\qquad |E(G_1)|=320.
\]

\section{The prescribed-vertex consequence}\label{sec:prescribed}

Hu, Xu and Zhuang asked whether, for every perfectly divisible graph
\(H\) and every vertex \(z\in V(H)\), there is a perfect division
\((A,B)\) of \(H\) such that \(z\in A\)
\cite[Problem 4.1]{HuXuZhuang2026}. To obtain the counterexamples in
Corollary~\ref{cor:intro-prescribed}, we delete the bisimplicial vertex
from each \(G_t\). Recall that \(N_{G_t}(v_t)\) induces the path
\(a-b-c-d\). Its end edges \(ab\) and \(cd\) give the two cliques
covering this neighbourhood. After deleting \(v_t\), it is the middle
edge \(bc\) whose endpoints lie outside the perfect part in every
perfect division.

\begin{proof}[Proof of Corollary \ref{cor:intro-prescribed}]
For \(t\geq1\), let
\[
 H_t=G_t-v_t.
\]
Since \(G_t\) is MNPD, its proper induced subgraph \(H_t\) is perfectly
divisible. Deleting \(v_t\) leaves every copy of \(F\) intact, so
\(\om(H_t)=\om(G_t)=3\). Any perfect division \((A,B)\) of \(H_t\)
therefore restricts to a perfect division of each copy of \(F\).
Lemma~\ref{lem:rooted-factor} thus puts all vertices of \(R_t\) in \(A\).
The odd holes \(C_b\) and \(C_c\) also survive the deletion. The only
vertex of \(C_b\) outside \(R_t\) is \(b\), so \(b\) must lie in \(B\).
Likewise \(C_c\) forces \(c\in B\). Thus neither endpoint of the edge
\(bc\) can belong to the perfect part.

The orders \(|V(H_t)|=|V(G_t)|-1\) increase strictly with \(t\), so the
graphs \(H_t\) are pairwise nonisomorphic.
\end{proof}

Hu, Xu and Zhuang used the edge \(bc\) to show that a perfectly
divisible graph can fail to be perfectly weight divisible, answering
their Problem~4.2 \cite[Section~4]{HuXuZhuang2026}. For a positive
integral vertex weighting \(h\), let \(\om_h(J)\) be the maximum total
weight of a clique in \(J\). An \emph{\(h\)-perfect division} of \(J\)
is a partition \((A,B)\) with \(J[A]\) perfect and
\(\om_h(J[B])<\om_h(J)\). A graph is \emph{perfectly weight divisible}
if every nonempty induced subgraph has such a division for every
positive integral vertex weighting.

Give \(b\) weight two and every other
vertex of \(H_t\) weight one. The neighbours of \(b\) in \(H_t\)
are \(a,u,c\), an independent set. Hence a clique through \(b\) has
weight at most three, and \(\om_h(H_t)=\om(H_t)=3\).
In any \(h\)-perfect division \((A,B)\) we would have
\(\om(H_t[B])\leq\om_h(H_t[B])<3\), since every weight is at least one.
Thus \((A,B)\) would also be an ordinary perfect division. But then \(b,c\in B\), and the
edge \(bc\) alone has weight three. Thus \(H_t\) is perfectly
divisible but not perfectly weight divisible. Hu, Xu and Zhuang also
note that replacing \(b\) by two adjacent vertices with the same
neighbours outside the pair gives an MNPD graph in which the pair is a
homogeneous set \cite[Section~4]{HuXuZhuang2026}.

\appendix

\section{The ten-vertex colouring bound and computational checks}\label{app:verification}\label{app:ten}

\paragraph{The colouring bound.}
Let \(J\) be a triangle-free graph on at most ten vertices, as in
Lemma~\ref{lem:rooted-factor}. If its maximum degree is at most three,
Brooks' theorem~\cite{Brooks1941} gives a 3-colouring component by
component, with paths and cycles coloured directly. Now suppose that
\(J\) has a vertex \(v\) of degree at least four. At most five vertices lie outside the closed
neighbourhood \(N[v]\). If they induce a bipartite graph, colour it with
two colours, give \(v\) one of those colours, and use a third colour
for its independent neighbourhood \(N(v)\).

If the graph outside \(N[v]\) is not bipartite, it must be an induced
5-cycle, and \(v\) has exactly four neighbours. For each vertex in \(N(v)\), its neighbours on the cycle form an
independent set. They therefore include at most one of the five pairs
of nonconsecutive cycle vertices.
The four vertices of \(N(v)\) account for at most four of these pairs.
Choose a pair that remains and let \(x\) be its common neighbour on the
cycle. Give \(x\) and \(v\) colour 3 and colour the
four-vertex path left on the cycle alternately 1 and 2. Among the
nonconsecutive pairs on the cycle, only the endpoints of this path
use both 1 and 2. No vertex of \(N(v)\) is adjacent to both endpoints,
so each has colour 1 or 2 available.

\paragraph{Computational checks.}
The ancillary programs check the rooted graph \(F\) and the first member
\(G_1\) of the family. Table~\ref{tab:verification-counts} gives the
counts for \(F\). All 508 perfect divisions place the root in the perfect
part. There are 16,377 proper induced subgraphs containing the root and
at least one edge. In each, the root can be placed in either prescribed
part.

\begin{table}[H]
\centering
\caption{Counts for the rooted graph \(F\).}
\label{tab:verification-counts}
\begin{tabular}{lr}
\toprule
quantity & count\\
\midrule
perfect divisions of \(F\) & 508\\
sets \(S\subseteq V(F)\) for which \(F[S]\) has an edge & 32,619\\
proper \(F[X]\) with \(r\in X\) and an edge & 16,377\\
\bottomrule
\end{tabular}
\end{table}

The programs use two tests for perfection. One searches for odd holes
and odd antiholes using Theorem~\ref{thm:spgt}. The other compares
chromatic and clique numbers on induced subgraphs. The two tests agree.
The programs also reconstruct \(Q_1\) and \(G_1\) and check the
expanded adjacency list and \texttt{graph6} encoding of \(G_1\).
The programs, input data, results and SHA-256 checksums are available in
the \href{https://zenodo.org/records/24211586/files/MNPD_verification_code_v1.0.0.zip?download=1}{Zenodo code archive}
at \url{https://doi.org/10.5281/zenodo.24211586}.
From the extracted archive root, the exhaustive check can be reproduced with
\begin{verbatim}
sh anc/supplement/verifier/reproduce.sh
\end{verbatim}
A separate check of the rooted graph using maximal triangle-free sets
and explicit partitions can be run with
\begin{verbatim}
python3 anc/check_rooted_witnesses.py
\end{verbatim}

\clearpage
\section{Adjacency list of the rooted graph}\label{app:rooted-adjacency}

Table~\ref{tab:rooted-adjacency} lists the neighbours of each vertex of
\(F\) in \texttt{graph6} label order. Vertex \(3\) is the root.

\begin{table}[H]
\centering
\small
\renewcommand{\arraystretch}{1.08}
\caption{Adjacency list of the 15-vertex rooted graph \((F,r)\).}
\label{tab:rooted-adjacency}
\begin{tabular}{c@{\quad}l@{\hspace{1.2cm}}c@{\quad}l}
\toprule
vertex & neighbours & vertex & neighbours\\
\midrule
0  & 1, 4, 5, 6, 9, 10, 11
& 8  & 1, 2, 3, 4, 11, 14\\
1  & 0, 2, 4, 5, 6, 8, 11
& 9  & 0, 2, 4, 5, 7, 10, 12, 14\\
2  & 1, 3, 5, 7, 8, 9, 12
& 10 & 0, 3, 9, 11, 12, 14\\
\(\mathbf{3}\) & 2, 5, 6, 7, 8, 10, 11, 12, 13, 14
& 11 & 0, 1, 3, 7, 8, 10, 13\\
4  & 0, 1, 8, 9, 12, 13, 14
& 12 & 2, 3, 4, 6, 9, 10, 13\\
5  & 0, 1, 2, 3, 9, 13, 14
& 13 & 3, 4, 5, 11, 12\\
6  & 0, 1, 3, 7, 12, 14
& 14 & 3, 4, 5, 6, 8, 9, 10\\
7  & 2, 3, 6, 9, 11
& &\\
\bottomrule
\end{tabular}
\end{table}

\section*{Data and code availability}

The programs and graph data for the checks in
Appendix~\ref{app:verification} are archived in Zenodo, version 1.0.0,
at \url{https://doi.org/10.5281/zenodo.24211586}.
The \href{https://zenodo.org/records/24211586/files/MNPD_verification_code_v1.0.0.zip?download=1}{downloadable code archive}
contains the input encodings and adjacency lists, verification programs,
expected outputs, reproduction instructions and SHA-256 checksums.


{\footnotesize
\begin{thebibliography}{99}
\setlength{\itemsep}{0pt}
\setlength{\parsep}{0pt}
\setlength{\topsep}{0pt}

\bibitem{Brooks1941}
R. L. Brooks, On colouring the nodes of a network,
Proc. Cambridge Philos. Soc. 37 (1941), 194--197.
\url{https://doi.org/10.1017/S030500410002168X}.

\bibitem{ChudnovskyRobertsonSeymourThomas2006}
M. Chudnovsky, N. Robertson, P. Seymour and R. Thomas,
The strong perfect graph theorem,
Ann. of Math. (2) 164 (1) (2006) 51--229.
\url{https://doi.org/10.4007/annals.2006.164.51}.

\bibitem{ChudnovskySeymour2023}
M. Chudnovsky and P. Seymour,
Even-hole-free graphs still have bisimplicial vertices,
J. Combin. Theory Ser. B 161 (2023) 331--381.
\url{https://doi.org/10.1016/j.jctb.2023.02.009}.

\bibitem{ChudnovskySivaraman2019}
M. Chudnovsky and V. Sivaraman,
Perfect divisibility and 2-divisibility,
J. Graph Theory 90 (1) (2019) 54--60.
\url{https://doi.org/10.1002/jgt.22367}.

\bibitem{Chvatal}
V. Chv\'atal, The minimality of the Mycielski graph,
in: R. A. Bari and F. Harary (eds.), Graphs and Combinatorics,
Lecture Notes in Mathematics 406 (1974), 243--246.
\url{https://doi.org/10.1007/BFb0066446}.

\bibitem{GodsilMcKay1978}
C. D. Godsil and B. D. McKay,
A new graph product and its spectrum,
Bull. Aust. Math. Soc. 18 (1) (1978) 21--28.
\url{https://doi.org/10.1017/S0004972700007760}.

\bibitem{Hoang2018}
C. T. Ho\`ang,
On the structure of (banner, odd hole)-free graphs,
J. Graph Theory 89 (4) (2018) 395--412.
\url{https://doi.org/10.1002/jgt.22258}.

\bibitem{Hoang2026}
C. T. Ho\`ang,
On the structure of perfectly divisible graphs,
Discrete Math. 349 (2) (2026) 114809.
\url{https://doi.org/10.1016/j.disc.2025.114809}.

\bibitem{HuXuZhuang2026}
Q. Hu, B. Xu and M. Zhuang,
Some properties of minimally nonperfectly divisible graphs,
arXiv:2603.01967v3 (2026).
\url{https://arxiv.org/abs/2603.01967v3}.

\bibitem{McKayPiperno2014}
B. D. McKay and A. Piperno,
Practical graph isomorphism, II,
J. Symbolic Comput. 60 (2014) 94--112.
\url{https://doi.org/10.1016/j.jsc.2013.09.003}.

\bibitem{Trotignon2015}
N. Trotignon,
Perfect graphs,
in L. W. Beineke and R. J. Wilson (eds.),
Topics in Chromatic Graph Theory,
Cambridge University Press, Cambridge, 2015, 137--160.
\url{https://doi.org/10.1017/CBO9781139519793.010}.

\bibitem{West2001}
D. B. West,
Introduction to Graph Theory,
2nd ed., Prentice Hall, Upper Saddle River, NJ, 2001.

\end{thebibliography}
}
\end{document}